\documentclass[a4paper,11pt,titlepage,leqno,twoside]{article} 

\usepackage[top=1.5in,bottom=2in,left=1.3in,right=1.3in,marginparwidth=1in]{geometry} 

\usepackage{mathrsfs}
\usepackage{amssymb}
\usepackage{amsmath}
\usepackage{amsfonts}
\usepackage{longtable}
\usepackage{paralist}
\usepackage[mathlines]{lineno}

\newtheorem{theorem}{Theorem}[section]
\newtheorem{Claim}[theorem]{Claim}
\newtheorem{definition}[theorem]{Definition}

\newtheorem{Fact}[theorem]{Fact}
\newtheorem{lemma}[theorem]{Lemma}

\newtheorem{remark}[theorem]{Remark}

\newenvironment{proof}
{\noindent \textsc{Proof.}}
{\hspace*{\fill}$\Box$\bigskip}

\newcommand{\dom}[1]{\ensuremath{\mathrm{dom}}(#1)}

\newcommand{\set}[2]{\ensuremath{\{#1 \,|\, #2 \}}}
\newcommand{\seq}[2]{\ensuremath{\langle #1 \,|\, #2 \rangle}}

\newcommand{\restr}[2]{\ensuremath{#1 \! \upharpoonright \! #2}}

\newcommand{\el}{\prec}

\newcommand{\sub}{\subseteq}

\newcommand{\bb}{\mathbb}

\newcommand{\beq}{\begin{equation}}
\newcommand{\eeq}{\end{equation}}
\newcommand{\brm}{\begin{remark}\begin{rm}}
\newcommand{\erm}{\end{rm}\end{remark}}
\newcommand{\mx}{\mathrm}
\newcommand{\bce}{\begin{compactenum}}
\newcommand{\ece}{\end{compactenum}}

\newcommand{\Prk}{\mathrm{Prk}}

\newcommand{\Add}{\mathrm{Add}}

\newcommand{\R}{\bb{R}}
\newcommand{\Q}{\bb{Q}}
\renewcommand{\P}{\bb{P}}
\newcommand{\TP}{{\sf TP}}

\newcommand{\M}{\bb{M}}
\newcommand{\x}{\times}
\newcommand{\A}{\mathscr{A}}

\newcommand{\GCH}{\sf GCH}

\newcommand{\MA}{\sf MA}

\newcommand{\no}{\noindent}

\newcommand{\uu}{\mathfrak{u}(\kappa)}

\newcommand{\cf}{\mathrm{cf}}

\newcommand{\D}{\bb{D}}

\newcommand{\Anti}{\mathrm{MaxAnti}}

\newcommand{\comp}{\mathrel{||}}

\makeatletter
\renewcommand\section{\@startsection {section}{1}{\z@}%
                                   {-3.5ex \@plus -1ex \@minus -.2ex}%
                                   {2.3ex \@plus.2ex}%
                                   {\noindent\center\textsc}}
\renewcommand\subsection{\@startsection{subsection}{2}{\z@}%
                                     {-3.25ex\@plus -1ex \@minus -.2ex}%
                                     {1.5ex \@plus .2ex}%
                                     {\noindent\center\textsc}}
\renewcommand\subsubsection{\@startsection{subsubsection}{3}{\z@}%
                                     {-3.25ex\@plus -1ex \@minus -.2ex}%
                                     {1.5ex \@plus .2ex}%
                                     {\noindent\center \textsc}}
\makeatother

\usepackage[pdftex]{hyperref}
\hypersetup{
    colorlinks=true, 
    linktoc=all,     
    linkcolor=blue,  
    allcolors=blue,
}

\usepackage{fancyhdr}
\begin{document}\thispagestyle{empty}

\begin{center}
\no {\large \MakeUppercase{Indestructibility of the tree property}\par}

\bigskip

\medskip

Radek Honzik, {\v S}{\'a}rka Stejskalov{\'a}

\medskip

\begin{footnotesize}
Charles University, Department of Logic,\\
Celetn{\' a} 20, Praha 1, 
116 42, Czech Republic\\
radek.honzik@ff.cuni.cz\\

\medskip

Version: \today
\end{footnotesize}
\end{center}

\medskip

\begin{quote}
{\bf Abstract.} In the first part of the paper, we show that if $\omega \le \kappa < \lambda$ are cardinals, $\kappa^{<\kappa} = \kappa$, and $\lambda$ is weakly compact, then in $V[\M(\kappa,\lambda)]$ the tree property at $\lambda = \kappa^{++V[\M(\kappa,\lambda)]}$ is indestructible under all $\kappa^+$-cc forcing notions which live in $V[\Add(\kappa,\lambda)]$, where $\Add(\kappa,\lambda)$ is the Cohen forcing for adding $\lambda$-many subsets of $\kappa$ and $\M(\kappa,\lambda)$ is the standard Mitchell forcing for obtaining the tree property at $\lambda = (\kappa^{++})^{V[\M(\kappa,\lambda)]}$. This result has direct applications to Prikry-type forcing notions and generalized cardinal invariants. In the second part, we assume that $\lambda$ is supercompact and generalize the construction and obtain a model $V^*$, a generic extension of $V$, in which the tree property at $(\kappa^{++})^{V^*}$ is indestructible under all $\kappa^+$-cc forcing notions living in $V[\Add(\kappa,\lambda)]$, and in addition by all forcing notions living in $V^*$ which are $\kappa^+$-closed and ``liftable'' in a prescribed sense (such as $\kappa^{++}$-directed closed forcings or well-met forcings which are $\kappa^{++}$-closed with the greatest lower bounds). 
\end{quote}

\section{Introduction}

Let $\lambda$ be an uncountable regular cardinal. We say that the \emph{tree property} holds at $\lambda$, and write $\TP(\lambda)$, if every $\lambda$-tree has a cofinal branch (equivalently, there are no $\lambda$-Aronszajn trees). Recently, there has been extensive research which studies the extent of the tree property at multiple successor cardinals (see for instance Neeman \cite{NEEMAN:tp} which contains a detailed bibliography on the subject), with the ultimate goal of checking whether it is consistent that the tree property holds at every regular cardinal greater than $\aleph_1$. Starting with a large cardinal $\lambda$, the method of proof is typically based on lifting an elementary embedding with critical point $\lambda$ through a forcing which collapses $\lambda$ to a successor cardinal, and on  application of criteria for a forcing notion adding or not adding new cofinal branches to \emph{existing} $\lambda$-trees (see Section \ref{sec:A} for examples of such criteria). With such criteria in place, one can argue that all $\lambda$-trees have cofinal branches in the final generic extension, and thus $\TP(\lambda)$ holds.

In this paper we study a related question and search for criteria for forcing notions adding or not adding \emph{new} $\lambda$-Aronszajn trees. Notice that if $\Q$ does not add new $\lambda$-Aronszajn trees over a model $V^*$ with $\TP(\lambda)$, then $\TP(\lambda)$ still holds in $V^*[\Q]$. We say that $\TP(\lambda)$ in $V^*$ is \emph{indestructible under $\Q$}. It is clear that identifying forcings which cannot add $\lambda$-Aronszajn trees (over some models) can be very helpful in constructing complex models with $\TP(\lambda$).

To give our paper more specific focus, we will work over the Mitchell model $V[\M(\kappa,\lambda)]$ (and its variants) in which $\lambda$ equals $\kappa^{++}$. Our main result is that if $\kappa = \kappa^{<\kappa}$ and $\lambda >\kappa$ is weakly compact, the tree property at $\lambda = \kappa^{++}$ in $V[\M(\kappa,\lambda)]$ is indestructible under all $\kappa^+$-cc forcing notions which live in an intermediate Cohen submodel $V[\Add(\kappa,\lambda)]$ of $V[\M(\kappa,\lambda)]$ (Theorem \ref{th:improve}). It is open whether the restriction of living in $V[\Add(\kappa,\lambda)]$ can be removed; but even with this restriction, the result is quite strong and applies to many Prikry-style forcing notions such as the vanilla Prikry forcing or Magidor forcing (see Section \ref{sec:app}) and to forcings which manipulate the generalized cardinal invariants (see Section \ref{sec:var}).

In Theorem \ref{th:main}, we assume that $\lambda$ is supercompact and integrate the method in Theorem \ref{th:improve} with a method of guessing all potential $\kappa^+$-closed and $\kappa^{++}$-liftable forcings (see Definition \ref{def:master}), obtaining a forcing $\R$ such that in $V[\R]$ the tree property at $\kappa^{++}$ is indestructible under all $\kappa^+$-cc forcings living in $V[\Add(\kappa,\lambda)]$, and also by all $\kappa^+$-closed, $\kappa^{++}$-liftable forcings living in $V[\R]$ such as $\kappa^{++}$-directed closed forcings or well-met $\kappa^{++}$-closed forcing with greatest lower bounds (see Definition \ref{def:wellmet}).

Let us conclude by a short discussion of how Theorem \ref{th:main} relates to existing results. The indestructibility for small $\kappa^+$-cc forcings of size $\kappa^+$ appeared in Unger \cite{UNGER:1} (with $\kappa = \omega$). A different model with indestructibility for $\kappa^{++}$-directed closed forcings is described in \cite{UNGER:1}, where $V^*$ is now constructed using a Laver function for a supercompact $\lambda$ to guess all $\kappa^{++}$-directed closed forcings (this idea goes back to Cummings and Foreman \cite{CUMFOR:tp}). We merged these results together into one model (which is a routine task), but importantly added all $\kappa^+$-cc forcings which  live in $V[\Add(\kappa,\lambda)]$.\footnote{Unger \cite{UNGER:1} allows $\Add(\kappa,\gamma)$, $\gamma$ arbitrary, as the single example of a $\kappa^+$-cc forcing of size larger than $\kappa^+$ which preserves $\TP(\kappa^{++})$ over $V[\M]$ (notice $\Add(\kappa,\gamma)$ lives already in $V$ and is $\kappa^+$-Knaster in the model in \cite{UNGER:1}).}

\subsection{Forcings which add $\kappa^{++}$-Aronszajn trees}

Let us give some examples of forcings which do add $\kappa^{++}$-Aronszajn trees to put into context the results in Theorem \ref{th:main} (see Section \ref{sec:prelim:M} for definitions of the forcing notions).

\bigskip

\noindent{\bf Example 1.}  Suppose $\omega \le \kappa < \lambda$ are cardinals, $\kappa^{<\kappa} = \kappa$, and $\lambda$ is weakly compact. It is known that the Mitchell forcing $\M = \M(\kappa,\lambda)$ forces $\TP(\lambda)$ while turning $\lambda$ into $\kappa^{++}$ of $V[\M]$. There is a projection onto $\M$ from the product $R^0 \x R^1$ where $R^0$ is equal to the Cohen forcing $\Add(\kappa,\lambda)$ and $R^1$ a $\kappa^+$-closed forcing (the ``term'' forcing). Since $R^0$ does not collapse cardinals, $R^1$ must do the collapsing and therefore $(\kappa^{++})^{V[R^1]} = \lambda$. As $2^\kappa = \kappa^+$ holds in $V[R^1]$, there is in $V[R^1]$ a special $(\kappa^{++})^{V[R^1]}$-Aronszajn tree $T$. The $\kappa^+$-Knaster forcing $R^0$ in $V[R^1]$ cannot add a cofinal branch to $T$, and therefore $T$ is a special $(\kappa^{++})^{V[R^1 \x R^0]}$-Aronszajn tree in $V[R^0 \x R^1]$. An analysis of $\M$ and $R^0 \x R^1$ shows that $R^0 \x R^1 \mbox{ is equivalent to } \M * \dot{\D}$, where $\dot{\D}$ is forced by $\M$ to be $\kappa$-closed, $\kappa^+$-distributive and $\kappa^{++}$-cc  (see \cite{ABR:tree}). By the discussion above, $\dot{\D}$ adds a (special) $\kappa^{++}$-Aronszajn tree.

\brm
The standard forcing notion for adding $\square_{\kappa^+}$ with conditions of size at most $\kappa$ has the same properties as $\dot{\D}$ over $V[\M]$; see \cite{UNGER:1} for more details regarding this forcing.
\erm

\noindent {\bf Example 2.} Let us work over $L$ and let $\lambda$ be weakly compact and $\kappa < \lambda$ regular. It is known that adding a single Cohen subset of $\lambda$ by $\Add(\lambda,1)$ destroys weak compactness of $\lambda$ while preserving its Mahloness. Since for an inaccessible $\lambda$, the existence of a $\lambda$-Aronszajn tree is equivalent to $\lambda$ not being weakly compact, it follows that $\Add(\lambda,1)$ adds a $\lambda$-Aronszajn tree $T$. Now consider the forcing $\M \x \Add(\lambda,1)$, where $\M = \M(\kappa,\lambda)$. We just argued that there is a $\lambda$-Aronszajn tree $T$ in $L[\Add(\lambda,1)]$. Since $\M$ is $\lambda$-Knaster in $L[\Add(\lambda,1)]$, it cannot add a cofinal branch through $T$, and therefore $T$ is a $\lambda$-Aronszajn tree in $L[\Add(\lambda,1) \x \M] = L[\M \x \Add(\lambda,1)]$ (note that $\lambda = (\kappa^{++})^{L[\M \x \Add(\lambda,1)]}$). It follows that $\Add(\lambda,1)$ is a $\kappa$-closed, $\kappa^{++}$-distributive forcing of size $\kappa^{++}$ in $L[\M]$ which adds a $\kappa^{++}$-Aronszajn tree $T$ (notice that since $\lambda$ is still Mahlo in $L[\Add(\lambda,1)]$, the model $L[\Add(\lambda,1)][\M]$ has no special $\lambda$-Aronszajn trees, and therefore $T$ must be non-special). 

\bigskip

\noindent
{\bf Example 3.} Let $P$ be the generalization to $\aleph_2$ of Jech's forcing for adding an $\aleph_1$-Souslin tree (see also Kunen \cite{KUNENsat} for details of this generalization) and let us consider this forcing over any model in which $2^{\aleph_1} = \aleph_2$ holds. $P$ is an $\aleph_1$-closed, $\aleph_2$-strategically closed forcing of size $\aleph_2$ which adds an $\aleph_2$-Aronszajn tree. If we generalize Jech's forcing to any $\kappa^{++} \ge \aleph_2$ with $2^{\kappa^+} = \kappa^{++}$ we get there is an $\aleph_1$-closed, $\kappa^{++}$-strategically closed forcing of size $\kappa^{++}$ which adds a $\kappa^{++}$-Aronszajn tree.

\brm Note that if we also consider forcings which collapse cardinals then any forcing which collapses $2^\kappa$ to $\kappa^+$ adds a special $\kappa^{++}$-Aronszajn tree.
\erm

\section{Preliminaries}\label{sec:prelim}

\subsection{Mitchell forcing and its variants}\label{sec:prelim:M}

In Section \ref{sec:Mitchell}, we shall use the standard Mitchell forcing which we now review for the benefit of the reader. Let $\omega \le \kappa < \lambda$ be cardinals, $\kappa^{<\kappa} = \kappa$ and $\lambda$ inaccessible. We define the standard Mitchell forcing $\M(\kappa,\lambda)$ as follows. Let $\P$ be the Cohen forcing $\Add(\kappa,\lambda)$, where we identify conditions in $\Add(\kappa,\lambda)$ with partial functions of size $<\kappa$ from $\kappa \x \lambda$ into $2$. For $\alpha < \lambda$, we write $\P_\alpha$ to denote the restriction of $\P$ to coordinates below $\alpha$ (we write $\restr{p}{\alpha}$ for the restriction of $p \in \P$ to $\P_\alpha$).

\begin{definition}\label{def:Morg}
Conditions in $\M(\kappa,\lambda)$ are pairs $(p^0,p^1)$ such that $p^0 \in \P$ and $p^1$ is a function with domain $\dom{p^1} \sub \lambda$ of size at most $\kappa$. For $\alpha$ in the domain of $p^1$, $p^1(\alpha)$ is a $\P_\alpha$-name and \beq 1_{\P_\alpha} \Vdash p^1(\alpha) \in \Add(\kappa^+,1)^{V[\P_\alpha]}.\eeq The ordering is defined as follows: $(p^0,p^1) \le (q^0,q^1)$ iff $p^0 \le q^0$ in $\P$ and the domain of $p^1$ extends the domain of $q^1$, and for every $\alpha \in \dom{q^1}$, \beq \restr{p^0}{\alpha} \Vdash_{\P_\alpha} p^1(\alpha) \le q^1(\alpha).\eeq
\end{definition}

If $\kappa,\lambda$ are understood from the context, we write just $\M$. For $\alpha < \lambda$, let us denote by $\M_\alpha$ the natural truncation of $\M$ to $\alpha$ (we write $\restr{(p^0,p^1)}{\alpha}$ for the restriction of $(p^0,p^1)$ to $\M_\alpha$).

Using the Abraham's analysis (see \cite{ABR:tree}), there is a projection onto $\M$ from the product $R^0 \x R^1$ where $R^0 = \P$ is $\kappa^+$-Knaster (under the assumption $\kappa^{<\kappa} = \kappa$) and $R^1$ is $\kappa^+$-closed (the ``term'' forcing).  This analysis also holds for the natural quotients of $\M$ and $R^0$ and $R^1$: in particular,
\beq\label{eq:m}\mbox{if $\alpha<\lambda$ is inaccessible, then there is a projection onto $\M/\M_{\alpha}$ from $R^0_{\alpha} \x R^1_{\alpha}$,}\eeq where, under relevant assumptions, the forcing $R^0_{\alpha}$ is $\kappa^+$-Knaster and $R^1_{\alpha}$ is $\kappa^+$-closed in $V[\M/\M_{\alpha}]$.

In Section \ref{sec:more}, we will use a modification of $\M$, which we denote $\M'$, and define a forcing $\R$ closely following Abraham \cite{ABR:tree}. Let us first define $\M'$:
\medskip{}
\begin{definition} \label{def:Mmod}
Let $\kappa<\lambda$ be as in Definition \ref{def:Morg}. The forcing $\M' = \M'(\kappa,\lambda)$ is defined exactly as $\M(\kappa,\lambda)$ with the following modifications:
\bce[(i)]
\item The domain of the functions in $\Add(\kappa,\lambda)$ is composed only of successor ordinals, i.e. $p^0 \in \Add(\kappa,\lambda)$ iff $p^0$ is a partial function from $\kappa \x \mx{SuccOrd}(\lambda)$ into $2$ of size $<\kappa$, where $\mx{SuccOrd}(\lambda)$ denotes the successor ordinals below $\lambda$. We denote this forcing $\P'$ and $\P'_\alpha$ for its restrictions, $\alpha<\lambda$.
\item The domain of $p^1$ in Definition \ref{def:Morg} is composed only of successor cardinals below $\lambda$.
\ece
\end{definition}

It is easy to check that this change has no material effect on the basic properties of $\M$ such as the product analysis by means of $R^0 \x R^1$ (where $R^0$ is now $\P'$): it is a technical device which enables easier factoring of $\R$ at limit cardinals.\footnote{The change of the domain of $p^1$ \emph{does} have an effect on the properties of $\M$ with respect to forcing the \emph{approachability property} or its negation (see \cite{8fold}). However, it has no effect for the tree property argument.}

Technically speaking, the definition of $\R = \R(\kappa,\lambda)$ is by defining forcings $\R_\alpha$ by induction on $\alpha < \lambda$ using the truncations $\M'_\alpha$, with $\R$ being equal to $\R_\lambda$ (we will write $\restr{(p^0,p^1,p^2)}{\alpha}$ to denote the restriction of $(p^0,p^1,p^2)$ to $\R_\alpha$). We will state the definition without making the induction explicit in the interest of brevity. 

Let $f^L: \lambda \to H(\lambda)$ be a Laver function for a supercompact cardinal $\lambda$. Let us first define a certain set $\A$ of inaccessible cardinals below $\lambda$. An inaccessible cardinal $\alpha < \lambda$ is in $\A$ if and only if $f^L(\alpha)$ is an $\R_\alpha$-name $\dot{Q}$ for a forcing notion and \beq 1_{\R_\alpha} \Vdash \dot{Q}\mbox{ is $\kappa^{+}$-closed}.\eeq

\begin{definition}\label{def:R}
Conditions in $\R = \R(\kappa,\lambda)$ are triples $(p^0,p^1,p^2)$ such that $(p^0,p^1)$ is in $\M'(\kappa,\lambda)$ and $p^2$ is a function with domain $\dom{p^2}$ of size at most $\kappa$ such that every $\alpha \in \dom{p^2}$ is an element of $\A$. For such an $\alpha$, $p^2(\alpha)$ is an $\R_\alpha$-name and \beq 1_{\R_\alpha} \Vdash p^2(\alpha) \in \dot{Q},\eeq where $\dot{Q}$ equals $f^L(\alpha)$. The ordering is defined as follows: $(p^0,p^1,p^2) \le (q^0,q^1,q^2)$ iff $(p^0,p^1) \le (q^0,q^1)$ in $\M'(\kappa,\lambda)$, the domain of $p^2$ extends the domain of $q^2$ and for every $\alpha \in \dom{q^2}$, \beq \restr{(p^0,p^1,p^2)}{\alpha} \Vdash_{\R_\alpha} p^2(\alpha) \le q^2(\alpha).\eeq
\end{definition}

\begin{Fact}\label{lm:Laver-cc}
$\R$ is $\lambda$-cc.
\end{Fact}

\begin{proof}
This is a standard argument using the fact that all forcings guessed by $f^L$ are elements of $H(\lambda)$ (and are therefore $\lambda$-cc) and $\lambda$ is a Mahlo cardinal, and hence the direct limits of $\R$ form a stationary set (see for more details \cite{JECHbook}).
\end{proof}

The forcing $\R$ shares some important properties with $\M$ and $\M'$, in particular there are forcings $R^0$ and $R^1$, with $R^0 = \P'$ being $\kappa^+$-Knaster (if $\kappa^{<\kappa} = \kappa$) and $R^1$ being $\kappa^+$-closed (the ``term'' forcing), such that there is a projection onto $\R$ from the product $R^0 \x R^1$. This analysis also holds for the natural quotients of $\R$ and $R^0$ and $R^1$: 
\beq\label{eq:q}\mbox{if $\alpha$ is in $\A$, then there is a projection onto $\R/\R_{\alpha+1}$ from $R^0_{\alpha+1} \x R^1_{\alpha+1}$,}\eeq where (under relevant assumptions) $R^0_{\alpha+1}$ is $\kappa^+$-Knaster and $R^1_{\alpha+1}$ is $\kappa^+$-closed in $V[\R/\R_{\alpha+1}]$. For completeness, let us define $R^1$: it consists of all conditions in $\R$ of the form $(\emptyset,p^1,p^2)$ (notice that $\emptyset$ is the weakest condition in $\P'$). The proofs of these properties are exactly as in \cite{ABR:tree}, Lemma 2.15 and 2.18.\footnote{The only difference between our $\R$ and the forcing of Abraham is that our forcing on the third coordinate of $\R$, with conditions written as $p^2$, is $\kappa^+$-closed -- which is sufficient for the present argument -- whereas Abraham considers forcings which are $\kappa^{++}$-directed closed, in preparation for his lifting argument.}

\brm
We can simplify the definition of $\R$ and obtain a forcing notion $\R^*$ which achieves the same results as $\R$ in Theorem \ref{th:main}: the conditions in $\R^*$ are pairs $(p^0,p^2)$ where $p^0$ is in $\P'$ and $p^2$ is defined as $p^2$ in $\R$ (with the obvious modification that $p^2(\alpha)$ is an $\R^*_\alpha$-name for $\alpha \in \A$). The ordering is as in $\R$: $(p^0,p^2) \le (q^0,q^2)$ iff $p^0 \le q^0$ in $\P'(\kappa,\lambda)$, the domain of $p^2$ extends the domain of $q^2$ and for every $\alpha \in \dom{q^2}$, \beq \restr{(p^0,p^2)}{\alpha} \Vdash_{\R^*_\alpha} p^2(\alpha) \le q^2(\alpha).\eeq The point is that unboundedly often $f^L(\alpha)$ will choose the Cohen forcing $\Add(\kappa^+,1)$ of $V[\R^*_\alpha]$, obviating the need for the extra coordinate $p^1$ (it follows that the tree property at $\lambda$ will hold by the same argument as for $\R$). We use the presentation with $\R$ to use the familiar setup of Abraham's paper \cite{ABR:tree}.
\erm
{~}
\subsection{Regular embeddings from elementary embeddings} \label{sec:lift}

Recall the following standard fact (see for instance \cite{Kunen:new}):

\begin{Fact}\label{fact:q}
Assume $P,Q$ are forcing notions, $G$ is a $P$-generic filter, and $i: P \to Q$ is a regular (also called ``complete'') embedding. Then $Q$ is equivalent to $P * Q/\dot{G}$, where $Q/\dot{G}$ is a $P$-name for a forcing notion with conditions \beq \set{q \in Q}{q \mbox{ is compatible with }i''G},\eeq with the ordering inherited of $Q$.\footnote{``$q$ is compatible with $i''G$'' is short for ``$(\forall p \in G)\; q$ is compatible with $i(p)$''.} We write $Q/G$ for the interpretation of $Q/\dot{G}$ in $V[G]$ and call $Q/G$ \emph{the quotient of $Q$ over $G$}.
\end{Fact}

We will summarize several observations which allow us to obtain regular embeddings from elementary embeddings. We start by relativizing the notion of a regular embedding to a pair of models of set theory. Let $M$ be a transitive model of set theory and $P \in M$ a forcing notion; we define $\Anti(P)^M$ as the set of all maximal antichains of $P$ which are elements of $M$.

\begin{definition}
Let $M$ and $N$ be two transitive models of set theory and $P \in M$ and $Q \in N$ partial orders. We say that $i: P \to Q$ is an \emph{$(M,N)$-regular embedding} if $i$ preserves the ordering and incompatibility and moreover for every $A \in \Anti(P)^M$, $i''A \in \Anti(Q)^N$.
\end{definition}

It is clear from the definition that if $i$ is an $(M,N)$-regular embedding, then whenever $G^*$ is $Q$-generic over $N$, then $G = i^{-1}{''}G^*$ is $P$-generic over $M$.

We will make use of the following fact:

\begin{Fact} \label{fact:regular}
Assume $j: M \to N$ is an elementary embedding with critical point $\lambda$ between a pair of transitive models of set theory and let $P \in M$ be a partial order such that $M \models$ ``$P$ is $\lambda$-cc''.
\bce[(i)]
\item \label{reg:i}The restriction $\restr{j}{P}: P \to j(P)$ is an $(M,N)$-regular embedding. In particular, if $G^*$ is $j(P)$-generic over $N$ and $G = j^{-1}{''}G^*$, then $j$ lifts to \beq \label{G-star} j: M[G] \to N[G^*].\eeq
\item \label{eq:ii} Moreover, if \beq \label{eq:inverse} \restr{j}{P} \in N \mbox{ and }\Anti(P)^N \sub \Anti(P)^M,\eeq then \begin{multline} \label{inside} N \models \restr{j}{P} \mbox{ is a regular embedding from $P$ into $j(P)$ and }\\ j(P) \mbox{ is equivalent to } P * j(P)/\dot{G}.\end{multline}
\ece
\end{Fact}

\begin{proof}
(i) By elementarity, $j$ preserves the ordering relation and compatibility between $P$ and $j(P)$. To argue for regularity, it suffices to show that if $A \in M$ is a maximal antichain in $M$ then $j''A \in N$ is a maximal antichain in  $j(P)$. This follows immediately by elementarity and the fact $j''A = j(A)$, which holds since $M \models$ ``$|A|<\lambda$'', and $j$ is the identity below $\lambda$.

(ii) First notice that $\restr{j}{P} \in N$ implies that $P = \dom{\restr{j}{P}} \in N$. To be able to carry out the quotient analysis from Fact \ref{fact:q} in $N$, it suffices to assume that $\restr{j}{P}$ is a regular embedding in $N$ which follows from the fact that it is an $(M,N)$-regular embedding and (\ref{eq:inverse}) holds.
\end{proof}

When $G$ is $P$-generic over $N$ and item (\ref{eq:ii}) of Fact \ref{fact:regular} applies, the definition of the quotient $j(P)/G$ is expressible in $N[G]$ and we can write:

\beq \label{eq:q2} j(P)/G = \set{p^* \in j(P)}{N[G] \models p^* \mbox{ is compatible with }j''G}.\eeq

\brm
Naively, one could try to weaken the assumption (\ref{eq:inverse}) and ask just for $P$ being an element of $N$ which is easier to ensure in general. With the assumption that $P \in N$, we could write $N[G]$; however this does not ensure that the quotient forcing $j(P)/G$ is an element of $N[G]$, blocking the final step which would show that $N[G^*]$ (where $G^*$ is as in (\ref{G-star})) can be decomposed as $N[G][H]$, where $H$ is $j(P)/G$-generic.\footnote{Notice for instance that without some extra assumptions such as (\ref{eq:inverse}), $P \in N$ does not in general guarantee $N[G] \sub N[G^*]$.}
\erm

\subsection{Forcings not adding branches to Aronszajn trees}\label{sec:A}

It is known that if $T$ is a tree of height $\mu$ with $\cf(\mu) = \kappa^+$ (with no limit on the size of the levels of $T$), then a forcing $P$ which is $\kappa^+$-square-cc\footnote{We say that $P$ is $\kappa^+$-square-cc if $P \x P$ is $\kappa^+$-cc.} does not add a new cofinal branch to $T$ (see for instance \cite{UNGER:1}).

However, if $T$ is a $\kappa^+$-Aronszajn tree, it is possible to weaken the assumption on $P$ to the usual (i.e. ``non-square'') chain condition, albeit at a smaller cardinal (see \cite{Kunen:new}, Exercise V.4.21):

\begin{Fact}\label{Kunen}
\bce[(i)]
\item Suppose $T$ is well-pruned $\kappa^+$-Aronszajn tree. Then for every $t \in T$, there is a level of the tree $T$ above $t$ which has size $\kappa$.
\item It follows that if $P$ is $\kappa$-cc, then it does not add a cofinal branch to $T$.
\ece
\end{Fact}

\begin{proof}
(i) For contradiction assume that all levels above $t$ have size $<\kappa$, and using Fodor's lemma, find a stationary set on which the nodes of the tree form a cofinal branch.

(ii) If $\dot{b}$ is a name for a cofinal branch through $T$, it can be used to build back in $V$ a subtree $S$ of $T$ of height $\kappa^+$ with levels of size $<\kappa$. By (i), $S$ must have a cofinal branch, and it is a cofinal branch through $T$ as well. Contradiction.
\end{proof}

We shall further use the following lemma due to Unger (see \cite[Lemma 6]{UNGER:1}), which generalizes an analogous result in \cite{JSkurepa} which is formulated for $\kappa = \omega$:

\begin{Fact}\label{f:spencer}
Let $\kappa,\lambda$ be cardinals with $\lambda$ regular and $2^\kappa \ge \lambda$. Let $P$ be $\kappa^+$-cc and $Q$ be $\kappa^+$-closed. Let $\dot{T}$ be a $P$-name for a $\lambda$-tree. Then in $V[P]$, forcing with $Q$ cannot add a cofinal branch through $T$.
\end{Fact}

\brm
In standard proofs in the literature for the tree property (for instance \cite{ABR:tree}, \cite{CUMFOR:tp}, or \cite{UNGER:1}), the assumption on $P$ is always that of the Knasterness (or square-cc). The reason is probably that before Fact \ref{f:spencer} was widely known, a $\kappa^+$-closed forcing (an analogue of $R^{*1}$ in the proof of Theorem \ref{th:improve}) was first applied to argue that a certain $\kappa^{++}$-tree $T$ (in fact an Aronszajn tree) does not get a cofinal branch. This $\kappa^+$-closed forcing typically collapses $\kappa^{++}$ to $\kappa^+$, which makes $T$ a tree whose height has cofinality $\kappa^+$. To argue that $P$ does not add cofinal branches to $T$ now, the stronger version with $\kappa^+$-square-cc was necessary. With Fact \ref{f:spencer}, we can first consider the $\kappa^+$-cc forcing, and only then deal with the collapsing $\kappa^+$-closed forcing.
\erm

\section{Indestructibility in the Mitchell model} \label{sec:Mitchell}

Let $\omega \le \kappa < \lambda$ be cardinals, $\kappa^{<\kappa} = \kappa$ and $\lambda$ weakly compact, and $\M  =\M(\kappa,\lambda)$ the Mitchell forcing. We will show that $\TP(\lambda)$ is indestructible over $V[\M]$ under all $\kappa^+$-cc forcings $\Q$ which live in $V[\Add(\kappa,\lambda)]$, an intermediate model between $V$ and $V[\M]$.

First note that it is enough to consider $\kappa^+$-cc forcings $\Q$ which have size $\kappa^{++}$ in $V[\M]$. This follows from the following more general lemma:

\begin{lemma}\label{lm:small}
\bce[(i)]
\item Suppose $\lambda^{<\lambda} = \lambda$ is a cardinal and $\Q$ is $\lambda$-cc. If $\Q$ adds a $\lambda$-Aronszajn tree, then there exists a regular $\lambda$-cc subforcing $\bar{\Q}$ of $\Q$ which adds a $\lambda$-Aronszajn tree. 
\item In particular, if no forcing notion of size at most $\lambda$ which is $\lambda$-cc adds a $\lambda$-Aronszajn tree, then no $\lambda$-cc forcing adds a $\lambda$-Aronszajn tree.
\ece
\end{lemma}

\begin{proof}
Let $\Q$ be a $\lambda$-cc forcing notion and assume that $q \in \Q$ forces that there is a $\lambda$-Aronszajn tree. Choose a large enough regular $\theta$ so that $\Q \in H(\theta)$. Let $M$ be an elementary submodel of $H(\theta)$ of size $\lambda$ closed under $<\lambda$-sequences which contains $\lambda$ as a subset and $\Q$ and $\lambda$ as elements. Let $\pi: M \to \bar{M}$ be the transitive collapse (note that the critical point of $\pi^{-1}$ is strictly above $\lambda$ since $\lambda+1$ is included in $M$). Let us denote $\pi(\Q)$ by $\bar{\Q}$. By elementarity, $\bar{\Q}$ is $\lambda$-cc in $\bar{M}$, but also in $V$ (because the image of any antichain of size $\lambda$ of $\bar{\Q}$ in $V$ would map via $\pi^{-1}$ to an antichain of size $\lambda$ in $\Q$, contradicting the $\lambda$-cc of $\Q$). Let $i$ be the restriction of $\pi^{-1}$ to $\bar{\Q}$. It is easy to check that \beq i: \bar{\Q} \to \Q \eeq is a regular embedding because by the closure of $M$ under $<\lambda$-sequences, every maximal antichain $A$ of $\bar{\Q}$ which exists in $V$ is an element of $M$ (and $\pi^{-1}(A) = \pi^{-1}{''}A$). Let $G$ be a $\Q$-generic over $V$ containing $q$; then $\bar{G} = \pi''G$ is $\bar{\Q}$-generic over $V$ and $\pi^{-1}$ lifts to \beq \pi^{-1}: \bar{M}[\bar{G}] \to H(\theta)[G].\eeq By our assumption, $H(\theta)[G]$ thinks there is a $\lambda$-Aronszajn tree, and by elementarity $\bar{M}[\bar{G}]$ must think the same. Let $T$ be a tree (which we construe as subset of $\lambda$) such that $(T$ is a $\lambda$-Aronszajn tree$)^{\bar{M}[\bar{G}]}$. Since $\pi^{-1}(T) = T$, $T$ is actually a $\lambda$-Aronszajn tree in $H(\theta)[G]$, and therefore in $V[G]$. It follows that $T$ cannot have a cofinal branch in $V[\bar{G}]\sub V[G]$, and therefore $(T$ is a $\lambda$-Aronszajn tree$)^{V[\bar{G}]}$. Thus $\bar{\Q}$ is a regular subforcing of $\Q$ which is $\lambda$-cc, has size $\lambda$ and adds a $\lambda$-Aronszajn tree.
\end{proof}

Let us now prove the main theorem of this section. Note that it is open whether Theorem \ref{th:improve} can be extended to include all $\kappa^+$-cc forcings $\Q$ living in $V[\M]$ (see Remark \ref{rm:open} and open question {\bf Q1} in Section \ref{sec:open}).

\begin{theorem}\label{th:improve}
Assume $\omega \le \kappa < \lambda$ are cardinals, $\kappa^{<\kappa} = \kappa$ and $\lambda$ is weakly compact. Let $\M$ be the standard Mitchell forcing $\M(\kappa,\lambda)$. 

Suppose $\Q \in V[\Add(\kappa,\lambda)]$ is $\kappa^+$-cc in $V[\Add(\kappa,\lambda)]$ (equivalently $\kappa^+$-cc in $V[\M]$), then $$V[\M * \dot{\Q}] \models \TP(\kappa^{++}).$$
In other words, the tree property at $\kappa^{++}$ is indestructible under any $\kappa^+$-cc forcing which lives in $V[\Add(\kappa,\lambda)]$.
\end{theorem}

\begin{proof}
By Lemma \ref{lm:small}, we can assume that $\Q$ has size at most $\kappa^{++}$ in $V[\M]$ and we can view it as a subset of $\kappa^{++}$ by using an isomorphic copy if necessary.

Let us first make the convention that we identify $\dot{\Q}$ with an $R^0 = \Add(\kappa,\lambda)$-name and we only consider conditions $(p,\dot{q})$ in $\M* \dot{\Q}$ in which $\dot{q}$ depends only on the $R^0$-information of $\M$.

Let $R^0 \x R^1$ denote the product from which there is a projection onto $\M$ (see Section \ref{sec:prelim:M} for more details).

Let us fix a weakly compact embedding $j: M \to N$ with critical point $\lambda$ such that $M$ has size $\lambda$, is closed under $<\lambda$-sequences and contains all relevant parameters, in particular the forcing $(R^0 \x R^1) * \dot{\Q}$. We can further assume that $j$ itself is an element of $N$ so that Fact \ref{fact:regular}(ii), in particular (\ref{eq:inverse}), applies to $M$, $N$ and $(R^0 \x R^1) * \dot{\Q}$. 

Let $(\tilde{G}^0 \x \tilde{G}^1) * h^*$ be $j((R^0 \x R^1) * \dot{\Q})$-generic filter over $V$ (and hence also over $M$). Since $(R^0 \x R^1) * \dot{\Q}$ is $\lambda$-cc, $j$ restricted to $(R^0 \x R^1) * \dot{\Q}$ is an $(M,N)$-regular embedding, and therefore $(\tilde{G}^0 \x \tilde{G}^1) * h^*$ generates an $M$-generic filter $(G^0 \x G^1) * h$ for $(R^0 \x R^1) * \dot{\Q}$, and $\tilde{G}^0 \x \tilde{G}^1$ generates an $M$-generic filter $G^*$ for $j(\M)$ and an $M$-generic filter $G$ for $\M$ such that $j$ lifts in $M[(\tilde{G}^0 \x \tilde{G}^1)*h^*]$ to: \beq \label{eq:decompose} j: M[G][h] \to N[G^* * h^*] = N[G][h][G_Q],\eeq where $G_Q$ is a generic filter for the quotient $Q = j(\M*\dot{\Q})/G*h$. See Fact \ref{fact:regular}(ii) for more details regarding the decomposition of $N[G^* * h^*]$ to $N[G][h][G_Q]$ in (\ref{eq:decompose}).

We will to show that over $N[G][h]$, \beq \label{eq:embed}\mbox{there is a projection onto } Q \mbox{ from } j(R^0 * \dot{\Q}))/(G^0 * h) \x R^1_\lambda,\eeq where $R^1_\lambda$ is the term forcing of $j(\M)/G$ (it is composed of conditions of the form $(0,p^{*1})$ -- we will write just $p^{*1}$). We will further show that $j(R^0 * \dot{\Q})/(G^0 * h)$ is $\kappa^+$-cc over $N[G][h]$ and $R^1_\lambda$ is $\kappa^+$-closed in $N[G]$ which will allow us to finish the argument.

\brm
Notice that it makes sense to consider the generic filter $G^0*h$ and write $N[G^0*h]$: it holds that $\dot{\Q}^G = \dot{\Q}^{G^0}$ since by our assumption $\dot{\Q}$ can be identified with an $R^0$-name.
\erm

Since $j$ is the identity on the conditions in $G$, we have \beq j''(G*h) = \set{(p,j(\dot{q}))}{p \in G \mbox{ and }\dot{q}^G \in h}.\eeq

Let us write explicitly the quotients we are going to use: \beq Q = \set{(p^*,\dot{q}^*) \in j(\M * \dot{\Q})}{N[G][h] \models \mbox{``}(p^*,\dot{q}^*) \mbox{ is compatible with }j''(G*h)\mbox{''}},\eeq where we can assume that $\dot{q}^*$ depends by elementarity only on $j(R^0)$. Further, \begin{multline} j(R^0 * \dot{\Q})/(G^0 * h) = \\ \set{(p^{*0},\dot{q}^*) \in j(R^0 * \dot{\Q})}{N[G^0][h] \models \mbox{``}(p^{*0},\dot{q}^*) \mbox{ is compatible with } j''(G^0 * h)\mbox{''}}.\end{multline} Lastly, \beq R^1_\lambda = \set{(0,p^{*1})}{N[G] \models \mbox{``}(0,p^{*1}) \mbox{ is compatible with }j''G = G\mbox{''}}.\eeq

Let us define a function $\pi: j(R^0 * \dot{\Q}))/(G^0 * h) \x R^1_\lambda \to Q$ by \beq \pi((p^{*0},\dot{q}^*),p^{*1}) = (p^*,\dot{q}^*),\eeq where $p^* = (p^{*0},p^{*1})$.

\begin{Claim}\label{pi}
$\pi$ is a projection from $j(R^0 * \dot{\Q})/G^0 *h \x R^1_\lambda$ onto $Q$.
\end{Claim}

\begin{proof}
First notice that $\pi$ is correctly defined: if $(p^{*0},\dot{q}^*)$ is compatible with $j''(G^0 *h)$, and $(0,p^{*1})$ is compatible with $G$, then $(p^*,\dot{q}^*)$ is compatible with $j''(G*h)$.

If $((p^{*0},\dot{q}^*),p^{*1}) \le ((r^{*0},\dot{s}^*),r^{*1})$, then clearly $p^* \le r^*$; moreover, $p^{*0} \Vdash \dot{q}^* \le \dot{s}^*$ implies $p^* \Vdash \dot{q}^* \le \dot{s}^*$ because $p^* = (p^{*0},p^{*1})$. It follows $(p^*,\dot{q}^*) \le (r^*, \dot{s}^*)$, and hence $\pi$ is order-preserving.

Suppose now $(p^*,\dot{q}^*) \le \pi((r^{*0},\dot{s}^*),r^{*1})= (r^*,\dot{s}^*)$ are given. We wish to find a condition extending $((r^{*0},\dot{s}^*),r^{*1})$ whose $\pi$-image extends $(p^*,\dot{q}^*)$. First notice that $p^* \Vdash \dot{q}^* \le \dot{s}^*$ implies $p^{*0} \Vdash \dot{q}^* \le \dot{s}^*$ because of our convention that $\dot{q}^*$ and $\dot{s}^*$ depend only on $R^0$. Now we use a standard trick with names: Consider conditions $(p^{*0}, \dot{q}^*)$ and $p^{*1'}$ where the name $p^{*1'}$ interprets as $p^{*1}$ below $p^{*0}$ and as $r^{*1}$ otherwise; then $((p^{*0},\dot{q}^*),p^{*1'})$ is as required.
\end{proof}

Finally, we need the following Claim:

\begin{Claim}\label{claim:cc}
\bce[(i)]
\item $R^1_\lambda$ is $\kappa^+$-closed in $N[G]$.
\item $j(R^0 * \dot{\Q})/G^0 *h$ is $\kappa^+$-cc over $N[G][h]$.
\item $\dot{\Q}^{G^0} * j(R^0 * \dot{\Q})/G^0 * \dot{h}$ is $\kappa^+$-cc over $N[G]$.
\ece
\end{Claim}

\begin{proof}
(i) This a standard fact (see for instance \cite{ABR:tree}).

(ii) By elementarity, \beq j(R^0 * \dot{\Q}) \mbox{ is $\kappa^+$-cc over $N$.}\eeq The term forcing $R^1$ is $\kappa^+$-closed over $N$. By Easton's lemma \beq \label{eq:ccc} j(R^0*\dot{\Q}) \mbox{ is $\kappa^+$-cc over $N[G^1]$.}\eeq Since $j$ restricted to $R^0 * \dot{\Q}$ is a regular embedding, $j(R^0 * \dot{\Q})$ factors over $N$ (and then also over $N[G^1]$) as \beq (R^0 * \dot{\Q}) * j(R^0 * \dot{\Q})/\dot{G}^0 * \dot{h},\eeq  where $j(R^0 * \dot{\Q})/\dot{G}^0 * \dot{h}$ is an $R^0 * \dot{\Q}$-name for the quotient. It follows by (\ref{eq:ccc}), and properties of two-step iterations, that over $N[G^1]$, the $\kappa^+$-cc forcing $R^0 * \dot{\Q}$ forces that $ j(R^0 * \dot{\Q})/\dot{G}^0 * \dot{h}$ is $\kappa^+$-cc. In particular, $j(R^0 * \dot{\Q})/G^0 *h$ is $\kappa^+$-cc over $N[G^1][G^0*h]$.

Since there is a natural projection from $(R^0 *\dot{\Q}) \x R^1$ onto $\M * \dot{\Q}$ (analogously to the projection $\pi$ mentioned above), it follows that $j(R^0 * \dot{\Q})/G^0 *h$ is $\kappa^+$-cc over $N[G][h]$ as desired (since the chain condition is preserved downwards).

(iii) Recall that $\dot{\Q}^{G^0}$ is $\kappa^+$-cc in $N[G]$ by our initial assumptions. By (ii) of the present Claim, $\dot{\Q}^{G^0}$ forces over $N[G]$ that $j(R^0 * \dot{\Q})/G^0 *\dot{h}$ is $\kappa^+$-cc. By general forcing properties this means the two-step iteration $\dot{\Q}^{G^0} * j(R^0 * \dot{\Q})/G^0 * \dot{h}$ is $\kappa^+$-cc in $N[G]$.
\end{proof}

With Claims \ref{pi} and \ref{claim:cc}, the theorem is proved as follows. Suppose for contradiction that there is in $M[G][h]$ a $\lambda$-Aronszajn tree $T$. By standard arguments, we can assume that $T$ is also in $N[G][h]$ (and is Aronszajn here), and $T$ has a cofinal branch in $N[G][h][G_Q]$ because of the lifted embedding $j$ in (\ref{eq:decompose}). We will argue that the forcing $Q$ cannot add a cofinal branch to $T$, which is a contradiction.

Working over $N[G][h]$, $\lambda = (\kappa^{++})^{N[G][h]}$ and therefore by Fact \ref{Kunen}(ii) and Claim \ref{claim:cc}(ii), $j(R^0 * \dot{\Q})/G^0 *h$ cannot add a cofinal branch to the $\lambda$-Aronszajn tree $T$. Using the fact that $2^\kappa = \lambda$ in $N[G]$, and Fact \ref{f:spencer} applied over $N[G]$ to the $\kappa^+$-closed forcing $R^1_\lambda$ and to the $\kappa^+$-cc forcing $\dot{\Q}^{G^0} * j(R^0 * \dot{\Q})/G^0 * \dot{h}$ (see Claim \ref{claim:cc}(iii)), it follows that $R^1_\lambda$ cannot add a cofinal branch to $T$ over a generic extension of $N[G][h]$ by the quotient $j(R^0 * \dot{\Q})/G^0 *h$. Thus, the product \beq \label{product} R^1_\lambda \x j(R^0 * \dot{\Q})/G^0 *h \eeq does not add cofinal branches to $T$ over $N[G][h]$. However, by Claim \ref{pi}, there is a projection onto the quotient $Q$ from the product (\ref{product}), and therefore $Q$ cannot add a cofinal branch to $T$. It follows that $T$ has no cofinal branch in $N[G][h][G_Q]$ which is the desired contradiction.
\end{proof}

\brm \label{rm:Prikry}
It is not in general possible to analyse the quotient $j(R^0 * \dot{\Q})/G^0 *h$ (and similar quotients) by arguing that it is equivalent to a two-step iteration $j(R^0)/G^0 * \dot{S}$, for some forcing $\dot{S}$ which deals with the quotient of $j(\dot{\Q})$ by $h$. For instance if $\kappa$ is  a Laver-indestructible supercompact cardinal and $\dot{\Q}$ is the name for the vanilla Prikry forcing, then over $N[G^0][h]$, the quotient $j(R^0 * \dot{\Q})/G^0 *h$ is not equivalent to any forcing of the form $j(R^0)/G^0 * \dot{S}$ because $j(R^0)/G^0$ (which is in fact equivalent to the Cohen forcing $j(R^0)$) collapses $\kappa$ to $\aleph_0$  because $h$ makes $\kappa$ singular with cofinality $\omega$.
\erm

\brm \label{rm:open}
There seems to be no obvious way to improve the Theorem \ref{th:improve} to consider $\kappa^+$-cc forcings $\Q \in V[\M]$: it was essential for the analysis in Theorem \ref{th:improve} to find a projection onto the quotient from the product of a $\kappa^+$-cc and $\kappa^+$-closed forcings, and this strategy does not work when $\Q$ lives in $V[\M]$. For all we know, it may be that the collapsing part of $\M$ introduces a $\kappa^+$-cc forcing of size $\kappa^{++}$ which destroys the tree property at $\kappa^{++}$ in $V[\M]$.
\erm

\subsection{Some applications}\label{sec:app}

\subsubsection{Prikry-type constructions}

Theorem \ref{th:improve} offers significantly easier proofs of the results which obtain the tree property at the second successor of a singular strong limit cardinal $\kappa$. Let us state some examples:

\begin{enumerate}[(1)]
\item Recall the result from \cite{CUMFOR:tp} by Cummings and Foreman: starting with $\kappa < \lambda$, where $\kappa$ is a Laver-indestructible supercompact cardinal and $\lambda$ is weakly compact, they construct a model in which the tree property holds at $\kappa^{++} = \lambda$ with $2^\kappa = \kappa^{++}$ and $\kappa$ is a singular strong limit cardinal with cofinality $\omega$. This result follows immediately as a corollary of Theorem \ref{th:improve} if we force with \beq \M(\kappa,\lambda) * \Prk^{V[\Add(\kappa,\lambda)]}(\dot{U}), \eeq where $\Prk^{V[\Add(\kappa,\lambda)]}(\dot{U})$ is the vanilla Prikry forcing and $\dot{U}$ is an $\Add(\kappa,\lambda)$-name for a normal measure on $\kappa$ in $V[\Add(\kappa,\lambda)]$ (notice that $V[\Add(\kappa,\lambda)]$ and $V[\M(\kappa,\lambda)]$ have the same subsets of $\kappa$).

\item The paper \cite{FHS1:large} generalises (and modifies) the construction in \cite{CUMFOR:tp} and obtains an arbitrarily large value of $2^\kappa$ with the tree property at $\kappa^{++}$ at a singular strong limit $\kappa$ with cofinality $\omega$ (starting with a Laver-indestructible supercompact cardinal $\kappa$ and a weakly compact $\lambda$ above $\kappa$). The result follows by the following straightforward application of Theorem \ref{th:improve}: 

Let $\delta \ge \lambda$ be a cardinal with cofinality at least $\kappa^+$ and consider the forcing \beq \label{eq:cof} \M * (\Add(\kappa,\delta) * \Prk^{V[\Add(\kappa,\delta)]}(\dot{U})),\eeq where $\Prk^{V[\Add(\kappa,\delta)]}(\dot{U})$ is the vanilla Prikry forcing as defined in $V[\Add(\kappa,\delta)]$. Now apply Theorem \ref{th:improve} with \beq \dot{\Q} = (\Add(\kappa,\delta) * \Prk^{V[\Add(\kappa,\delta)]}(\dot{U})).\eeq  Notice that $\Add(\kappa,\delta)$ is isomorphic to $\Add(\kappa,\lambda) \x \Add(\kappa,\delta)$ so that any normal measure $\dot{U}$ in $V[\Add(\kappa,\delta)]$ measures all subsets of $\kappa$ in $V[\M * \Add(\kappa,\delta)]$. In any generic extension by $V[\M*\dot{\Q}]$, $\kappa$ is a singular strong limit cardinal with cofinality $\omega$, the tree property holds at $\kappa^{++} = \lambda$ and $2^\kappa = \delta$.

\item Applications to Prikry-type constructions are not limited to countable cofinalities. Magidor's forcing, see for instance \cite{M:cof}, changes the cofinality of $\kappa$ to any regular cardinal below $\kappa$ with a forcing notion which is $\kappa^+$-cc (under the relevant assumptions, see \cite{M:cof}). Theorem \ref{th:improve} readily gives the following result:\footnote{This theorem appears independently in a preprint \cite{GP:gap} with the ``traditional proof'' following the method in \cite{FHS1:large}.}

\begin{theorem}
Suppose $\kappa <\lambda$ are cardinals, $\kappa$ is a Laver-indestructible supercompact cardinal and $\lambda$ is weakly compact. Let $\mu < \kappa$ be a regular cardinal and $\delta \ge \lambda$ a cardinal with cofinality at least $\kappa^+$. Then there is a forcing notion $\P$ such that in $V[\P]$ only the cardinals in the interval $(\kappa^+,\lambda)$ are collapsed, $\kappa$ is a singular strong limit cardinal with cofinality $\mu$, $2^\kappa = \delta$ and the tree property holds at $\kappa^{++} = \lambda$.
\end{theorem}

To argue for the theorem, consider the forcing in (\ref{eq:cof}) but replace the vanilla Prikry forcing with the Magidor forcing (as defined in $V[\Add(\kappa,\delta)]$). Notice that the increasing Mitchell-sequence of normal measures on $\kappa$ of length $\mu$ required to define Magidor forcing exists in $V[\Add(\kappa,\delta)]$ by our assumption on Laver-indestructibility of $\kappa$ and this sequences has the desired properties also in $V[\M]$.

\end{enumerate}

\subsection{Generalized cardinal invariants}\label{sec:var}

Theorem \ref{th:improve} can be used to control certain combinatorial characteristics of $2^\kappa$ if they are controlled by the Cohen part $\Add(\kappa,\lambda)$ of the forcing $\M$ (and its variants). 

\begin{enumerate}[(1)]
\item Kunen's observation in Fact \ref{Kunen} with the lifting and quotient analysis as in Theorem \ref{th:improve} (applied just to $\Q$) readily implies the following:

\begin{theorem}
It is consistent from large cardinals that $2^\omega$ is weakly inaccessible, the tree property holds at $2^\omega$ and $\MA$ (Martin's axiom) holds.
\end{theorem}

\begin{proof}
Assume $\lambda$ is weakly compact, and therefore it satisfies the tree property. Let $\Q$ be the standard finite-support ccc iteration of length $\lambda$ for forcing $2^\omega = \lambda$ and $\MA$. By the lifting argument as in Theorem \ref{th:improve} (applied with a weakly compact embedding which suffices if $\Q$ is included in $V_\lambda$), it follows easily that every $\lambda$-tree has a cofinal branch in $V[\Q]$, and therefore $\lambda$ has the tree property in $V[\Q]$.
\end{proof}

\item Recall that if $\kappa$ is an infinite cardinal, then $\uu$ is the least cardinal such that there is a uniform ultrafilter on $\kappa$ with base of size $\uu$ ($\uu$ makes sense also for singular cardinals, see \cite{GS:small}, and its size is always at least $\kappa^+$). Using certain forcings $\dot{\Q}$ -- which appear in \cite{GS:pol, GS:small, DS:graph, F:u} and which are $\kappa^+$-cc --, one can use Theorem \ref{th:improve} to construct models in which the tree property holds at $\kappa^{++}$, with $2^\kappa> \kappa^+$, $\uu = \kappa^+$, and where $\kappa$ is either inaccessible (in fact supercompact), or singular strong limit cardinal with cofinality $\omega$ (the singular case requires some extra work in addition to methods of Theorem \ref{th:improve}). These results are stated with more details in a separate paper \cite{HS:both-u}.
\end{enumerate}

\section{More indestructibility}\label{sec:more}

In this section, we will use the Laver function $f^L: \lambda \to H(\lambda)$ to guess all possible $\kappa^+$-closed forcing notions: this will improve the degree of indestructibility in the final model. In order to state the theorem with a bit more generality, let us introduce a definition of a $\lambda$-liftable forcing notion.\footnote{We thank James Cummings for a valuable discussion regarding this definition.}

Before we state the definition, recall that if $\lambda \le \theta$ are regular cardinals, we can define the notion of a closed unbounded set in $[H(\theta)]^{<\lambda}$: we say that $C \sub [H(\theta)]^{<\lambda}$ is closed unbounded if it is unbounded in the inclusion relation and the union of every $\sub$-increasing sequence of elements of $C$ of length $<\lambda$ is in $C$. Notice that if $|H(\theta)| = \theta$ (which will be the case for us), we can translate these concepts directly to the system $P_\lambda(\theta) = [\theta]^{<\lambda}$: Fix a bijection $f:\theta \to H(\theta)$. Then $C^*$ is closed unbounded in $P_\lambda(\theta)$ if and only if $C = \set{f''x}{x \in C^*}$ is closed unbounded in $[H(\theta)]^{<\lambda}$. Suppose $U$ is a normal measure on $P_\lambda(\theta)$, $j: V \to M$ is the derived elementary embedding, and $C^*$ is closed unbounded in $P_\lambda(\theta)$. Since $U$ extends the filter generated by closed unbounded sets, it contains $C^*$ and by the standard properties of the ultrapower of $V$ by $U$, $j''\theta \in j(C^*)$, but also \beq \label{jf} j(f)''(j''\theta) = j''H(\theta) \in j(C).\eeq
The property (\ref{jf}) will be useful in what follows.

\begin{definition}\label{def:master}
Let $\lambda$ be a regular cardinal and $\P$ a forcing notion. We say that $\P$ is $\lambda$-\emph{liftable} if $\P$ is $\lambda$-distributive and for a sufficiently large regular $\theta$ with $\P \in H(\theta)$ there is a closed unbounded set $C$ in $[H(\theta)]^{<\lambda}$ of elementary substructures $N \el H(\theta)$ which contain $\P$ and satisfy the following condition:
\bce[(*)]
\item Let $\pi_N: N \to \bar{N}$ denote the transitive collapse. For every $\pi_N(\P)$-generic filter $\bar{G} \in V$  over $\bar{N}$ there is a condition $p_N \in \P$ such that \beq \label{eq:master} p_N \Vdash_\P \pi^{-1}_N{''}\bar{G} \sub \dot{G}_\P,\eeq where $\dot{G}_\P$ is a name for a $\P$-generic filter. We call $p_N$ a master condition.
\ece
\end{definition}

Let us make a few comments regarding the definition. First note that (\ref{eq:master}) is for separative forcing notions equivalent to $p_N$ being a lower bound of $\pi_N^{-1}{''}\bar{G}$. Let us mention that the condition regarding the closed unbounded set $C$ is relevant only if there are some generic filters $\bar{G}\in V$ for the collapsed structures $\bar{N}$. In the context of the intended applications in which $\lambda$ is a critical point of a (generic) elementary embedding, there will be many such structures, but in general this may not be the case.  Also note that we explicitly require as a part of the definition that $\P$ is $\lambda$-distributive (if $\lambda$ is a critical point of a (generic) elementary embedding, this will again follow for free).

The intuition behind Definition \ref{def:master} is to capture a uniform combinatorial property of having a master condition which applies to all $\lambda$-directed closed forcing notions, but also to other forcing notions such as the generalized Sacks forcing at $\lambda$ (it is known that it is not $\lambda$-directed closed).\footnote{See \cite{KANAMORIperfect} for the definition of the generalized Sacks forcing.} As we will see in the proof of Theorem \ref{th:main}(\ref{e}), being liftable is strong enough to carry out a master condition argument and argue that $\lambda$-liftable forcings preserve the tree property at $\lambda$.

However, we should emphasize that $\P$ being $\lambda$-liftable does not by itself ensure that any elementary embedding $j: V^* \to M^*$ (for some models $V^*,M^*$) with critical point $\lambda$ lifts through a $\P$-generic filter $G$ over $V^*$: Typically, it is necessary that $j''G$ is an element of $M^*$ because then the $j(\lambda)$-liftability of $j(\P)$ (given by elementarity) ensures that there is a $j(\P)$-generic filter which extends $j''G$, and thus $j$ can lift.\footnote{Recall that a single Cohen subset of $\lambda$ can destroy supercompactness of $\lambda$, so even $\lambda$-directed closure is not sufficient for lifting unless we perform some preparation.} See the proof of Theorem \ref{th:main}(\ref{e}) for more details.

\brm \label{rm:complete}
We discovered,\footnote{We thank M.~Habic for pointing out this connection to us.} after we formulated the definition of a $\lambda$-liftable forcing notion, that it bears resemblance to the notion of a \emph{complete} forcing notion (for $\lambda = \omega_1$) which was introduced by Shelah (see \cite[Chapter V]{SHELAHproper}) and which is relevant for lifting of embeddings in the context of countable models and proper forcings which do not add reals. It is an interesting question to what extent the notions of \emph{complete} and also \emph{subcomplete} forcing notions are relevant for $\lambda > \omega_1$ and whether they can be characterized by other means (analogously to Jensen's result who showed that the class of complete forcing notions is exactly the class of forcing notions which densely embed an $\omega_1$-closed forcing notion; see \cite{Singapore} which contains a review of the complete and subcomplete forcing notions).

\erm

Let us give a few examples of $\lambda$-liftable forcing notions. 

\begin{definition}\label{def:wellmet}
Let us say that a forcing notion $\P$ is \emph{well-met $\lambda$-closed with glb} if (i) any two compatible conditions $p,q$ in $\P$ have the greatest lower bound (glb) which we denote by $p \wedge q$, and (ii) any decreasing sequence of elements $\seq{p_\alpha}{\alpha < \delta}$ of length $\delta <\lambda$ has glb which we denote $\bigwedge_{\alpha<\delta}p_\alpha$. 
\end{definition}

For instance the generalized Sacks forcing at $\lambda$ (both product and iteration) is well-met $\lambda$-closed with glb.\footnote{Strictly speaking, this depends on the presentation of the forcing. The presentation of the generalized Sacks forcing at $\lambda$ in \cite{KANAMORIperfect} is not well-met, but there is an equivalent presentation which is well-met and $\lambda$-closed with glb: define a condition as a tree which contains a perfect tree according to the definition in \cite{KANAMORIperfect} (the usual presentation is therefore dense in this ordering). It is straightforward to check that $p \cap q$ is the greatest lower bound of compatible conditions $p,q$ and $\bigcap_{\alpha < \beta}p_\alpha$ is the greatest lower bound for a decreasing sequence of conditions of length $\beta <\lambda$.}

\begin{lemma} 
\bce[(i)]
\item All $\lambda$-directed closed forcings are $\lambda$-liftable.
\item All forcings which are well-met $\lambda$-closed with glb are $\lambda$-liftable.
\ece
\end{lemma}

\begin{proof} (i) We show that all submodels $N$ satisfy the property in Definition \ref{def:master}.  If $\bar{N}$ is any structure as in Definition \ref{def:master}, $\pi^{-1}_N{''}\bar{G}$ is a directed set of condition in $\P$ of size $<\lambda$ and therefore has a lower bound.

  (ii). We show that all submodels $N$ satisfy the property in Definition \ref{def:master}.  Let $\bar{N}$ be any structure as in Definition \ref{def:master} and let $\bar{G}$ be $\pi_N(\P)$-generic over $\bar{N}$; let us denote $\pi_N$ by $\pi$ and $|\pi(\P)|$ by $\mu$. Let us fix an enumeration $\seq{q_\alpha}{\alpha<\mu}$ of $\pi^{-1}{''}\bar{G}$. We will construct by induction a decreasing sequence of elements $\seq{p_\alpha}{\alpha<\mu}$ in $\P$ which satisfies \beq \label{well-met} \forall \alpha < \mu \; p_\alpha \comp \pi^{-1}{''}\bar{G},\eeq where $p_\alpha \comp \pi^{-1}{''}\bar{G}$ means that $p_\alpha$ is compatible with every element of $\pi^{-1}{''}\bar{G}$. In addition, we will make sure that every element of $\pi^{-1}{''}\bar{G}$ is eventually above some element in $\seq{p_\alpha}{\alpha<\mu}$.

Assume $\seq{p_\beta}{\beta<\alpha}$ is constructed and (\ref{well-met}) holds for all $p_\beta$, $\beta < \alpha$. We describe the construction of $p_\alpha$.

\emph{Successor stage $\alpha = \beta+1$.} Let $p_{\beta+1} = p_\beta \wedge q_{\beta+1}$ (this is correctly defined since by the induction assumption $p_\beta$ is compatible with every element in $\pi^{-1}{''}G$). We need to check that $p_{\beta+1}$ is compatible with $\pi^{-1}{''}G$. Let us fix any $q = \pi^{-1}(\bar{p})$, $\bar{p} \in G$. Since $q \wedge q_{\beta+1} \in \pi^{-1}{''}\bar{G}$, $p_\beta$ is compatible with $q \wedge q_{\beta+1}$ by the induction assumption (and $p_\beta \wedge q_{\beta+1} \wedge q$ is the greatest lower bound).

\emph{Limit stage $\alpha$.} Let us first set $p'_\alpha = \bigwedge_{\beta<\alpha} p_\beta$, and then $p_\alpha = p'_\alpha \wedge q_\alpha$ (this is correctly defined since $q_\alpha$ is by induction compatible with every $p_\beta$, $\beta<\alpha$, and therefore must be compatible with $p'_\alpha$). Following the argument for the successor stage with the same notation, since $q \wedge q_\alpha \in \pi^{-1}{''}\bar{G}$, $p_\alpha'$ is compatible with $q \wedge q_\alpha$ (otherwise some $p_\beta$, $\beta < \alpha$, would not be compatible $q \wedge q_\alpha$ contradicting the induction assumption), with $p_\alpha' \wedge q_\alpha \wedge q$ being the greatest lower bound.

Let $\seq{p_\alpha}{\alpha<\mu}$ be the final sequence. By construction, for each $\alpha < \mu$, $p_\alpha \le q_\alpha$, and therefore it follows that $p_N = \bigwedge_{\alpha < \mu}p_\alpha$ is a master condition in the sense of (\ref{eq:master}) as desired.
\end{proof}

\brm Definition \ref{def:master} also applies in the context of Laver indestructibility of supercompactness: by preparing for all $\alpha$-liftable forcing notions (which are also $\alpha$-strategically closed to ensure sufficient distributivity of the tails of the Laver preparation) below a supercompact $\lambda$, one can get indestructibility of supercompactness for more forcing notions (such the generalized Sacks forcing at $\lambda$). However, note that $\lambda$-liftable forcing notions cannot include all $\lambda$-closed forcing notions because it is known that there are $\lambda$-closed forcing notions which can destroy weak compactness (such as the forcing for adding a $\lambda$-Kurepa tree\footnote{See \cite{CUMhandbook} for more details about this forcing; notice that the forcing is not well-met.}). Compare also with Jensen's result mentioned in Remark \ref{rm:complete}.
\erm

\brm
There are forcing  notions which can be lifted using a master condition argument, but are not necessarily $\lambda$-liftable. The point is that being $\lambda$-liftable means that \emph{any} normal ultrafilter on $P_\lambda(\theta)$ (in some outer model of $V$) contains the closed unbounded set of substructures (modulo some bijection between $H(\theta)$ and $\theta$) mentioned in Definition \ref{def:master}; this allows a uniform statement of the definition. In certain situations it is possible to choose a normal ultrafilter which contains the set of substructures mentioned in Definition \ref{def:master} even when the set is just stationary (an example is the forcing -- in the context a normal measure on $\lambda$, or equivalently on $P_\lambda(\lambda)$ -- for shooting a club through a suitable stationary set $S \sub \lambda$; if $C$ is the generic club through $\lambda$ contained in $S$ then $C \cup \{\lambda\}$ is a legitimate condition for shooting a club through $j(S)$, and hence a master condition, whenever $S$ is in the normal measure on $\lambda$).
\erm

We will prove the following theorem:

\begin{theorem}\label{th:main}
$\GCH$. Assume $\omega \le \kappa < \lambda$ are cardinals, $\kappa^{<\kappa} = \kappa$ and $\lambda$ is supercompact. Let $\R$ be the forcing from Definition \ref{def:R}. Suppose $\Q \in V[\R]$ satisfies any of the following conditions:
\bce[(a)]
\item \label{a} $\Q$ lives in $V[\R]$ and is $\kappa^+$-cc with size at most $\kappa^+$.
\item \label{b} $\Q$ lives in $V[\Add(\kappa,\lambda)]$ and is $\kappa^+$-cc in $V[\Add(\kappa,\lambda)]$ (equivalently $\kappa^+$-cc in $V[\R]$).
\item \label{c} $\Q$ lives in $V[\R]$ and is $\kappa^+$-distributive with size at most $\kappa^+$.
\item \label{e} $\Q$ lives in $V[\R]$ and is $\kappa^+$-closed and $\kappa^{++}$-liftable.
\item \label{f} $\Q$ lives in $V[\R]$ and is $\kappa^{+++}$-distributive in $V[\R]$.
\ece
Then $$V[\R * \dot{\Q}] \models \TP(\kappa^{++}).$$
In other words, the tree property at $\kappa^{++}$ is indestructible under all forcings $\Q$ listed in (a)--(e).
\end{theorem}

\begin{proof}
We first prove items (\ref{b},\ref{e},\ref{f}), leaving (\ref{a},\ref{c}) for the end.

(\ref{b}) It is easy to check that the analysis in the proof of Theorem \ref{th:improve} applies to $\R$ instead of $\M$. The key point is the existence of the projection from the product $R^0 \x R^1$ onto $\R$, with $R^0$ being isomorphic to $\Add(\kappa,\lambda)$ and $R^1$ being a $\kappa^+$-closed term forcing, and the existence of the quotient projection from Claim \ref{pi}. Notice that for this case, $\lambda$ may be just weakly compact.

(\ref{e}) For this case, we really need that $\lambda$ is supercompact. Let $\theta$ be a regular cardinal greater or equal to $|(R^0 \x R^1) * \dot{\Q}|$, where $R^0$ and $R^1$ are specified below Definition \ref{def:R}. Choose $\theta$ large enough so that for any $\R$-generic $G$, $(\dot{\Q})^G$ is $\lambda$-liftable in $V[G]$ for this $\theta$ in the sense of Definition \ref{def:master}. Choose $j: V \to M$ with critical point $\lambda$ so that $M$ is closed under $\theta$-sequences; ensure moreover that $j(f^L)(\lambda) = \dot{\Q}$. Let $\tilde{G}^0 \x \tilde{G}^1$ be $j(R^0 \x R^1 )$-generic over $V$. In $V[\tilde{G}^0 \x \tilde{G}^1]$, $j$ lifts to \beq j: V[G] \to M[G^*] = M[G][g][H],\eeq where $G$ is $\R$-generic, $g$ is $\dot{\Q}^G$-generic and $H$ is $j(\R)/(G*g)$-generic over $M[G][g]$. This follows by our definition of $\R$ which ensures that in $j(\R)$ only the forcing $\dot{\Q}$ appears at stage $\lambda$.

In the next step, we wish to lift $j$ further to $V[G][g]$ mimicking the usual master condition argument using the abstract criterion of $\Q = (\dot{\Q})^G$ being $\lambda = (\kappa^{++})^{V[G]}$-liftable in $V[G]$. By elementarity, this implies \beq \label{eq:j(Q)} \mbox{$j(\Q)$ is $j(\lambda)$-liftable in $M[G][g][H]$.}\eeq We will work in $M[G][g][H]$ now, aiming to use (\ref{eq:j(Q)}). By our assumptions, $\Q$ is an element of $\bar{N} = H(\theta)^{V[G]} = H(\theta)^{M[G]}$. Consider $N = j''\bar{N}$; this structure is an element of $M[G][g][H]$ and contains $j(\Q)$ as an element ($N$ is an element of $M[G][g][H]$ because modulo a bijection $f$ in $V[G]$ between $H(\theta)^{V[G]}$ and $\theta$, it is expressible as $j''\theta$). It follows by the uniqueness of the transitive collapse that $j^{-1}$ restricted to $N$ is in $M[G][g][H]$ the transitive collapse of $N$ to $\bar{N}$ which maps $j(\Q)$ to $\Q$. By $\lambda$-liftability of $\Q$ in $V[G]$, there is a closed unbounded set $C$ of substructures in $[H(\theta)]^{<\lambda}$ of $V[G]$ which have a master condition in the sense of Definition \ref{def:master} with respect to $\Q$. Since $N$ is in $j(C)$ by the analysis in (\ref{jf}) applied in $V[G]$ with a bijection $f$ in $V[G]$, we conclude that $N$ is a substructure to which (*) of Definition \ref{def:master} applies with respect to $j(\Q)$. Since both $g$ and $j''g$ are elements of $M[G][g][H]$, and $\restr{j^{-1}}{N}$ is the transitive collapse map of $N$ to $\bar{N}$, we conclude by (\ref{eq:j(Q)}) that $j''g$ has a lower bound in $j(\Q)$: i.e., there is some $p_N $ in $j(\Q)$ which forces $j''g$ into the generic filter for $j(\Q)$.

Force over $V[\tilde{G}^0 \x \tilde{G}^1]$ to obtain a generic filter $h$ for $j(\Q)$ below the condition $p_N$. Now $j$ lifts in $V[\tilde{G}^0 \x \tilde{G}^1][h]$ further to \beq j: V[G][g] \to M[G][g][H][h],\eeq and we can finish the argument in the standard way, arguing -- using the $\kappa^+$-closure of $j(\Q)$ -- that the generic $H*h$ does not add cofinal branches to a hypothetical $\lambda$-Aronszajn tree in $M[G][g]$.

(\ref{f}) We state this case just for completeness. Such forcings do not add any new $\kappa^{++}$-trees, and therefore preserve the tree property at $\kappa^{++}$ over $V[\R]$.

We will indicate how to modify the construction to deal with the small forcings in (\ref{a},\ref{c}). The argument is similar to Unger's in \cite{UNGER:1} extended to deal with an uncountable regular $\kappa$ and the extra case (\ref{c}); we will review the argument here for the benefit of the reader. Taking an isomorphic copy of $\dot{\Q}$ if necessary, we can assume that $\dot{\Q}$ is forced to be an element of $H(\lambda)^{V[\R]}$. Let $j: V \to M$ be an elementary embedding\footnote{A weakly compact embedding suffices here.} with critical point $\lambda$ chosen so that \beq \label{triv} \mbox{$j(f^L)(\lambda)$ is the trivial forcing}\eeq and let $(\tilde{G}^0 \x \tilde{G}^1) * g$ be $j((R^0 \x R^1 ) *\dot{\Q})$-generic over $V$. In $V[(\tilde{G}^0 \x \tilde{G}^1) * g]$, $j$ lifts to \beq j: V[G] \to M[G^*] = M[G][H],\eeq where $G$ is $\R$-generic and $H$ is $j(\R)/G$-generic over $M[G]$. By our assumption on $\dot{\Q}$, $j(\dot{\Q}^G) = \dot{\Q}^G$ and therefore $j^{-1}{''}g = g$, and $j$ lifts to \beq \label{a:cof} j: V[G][g] \to M[G][H][g].\eeq By (\ref{eq:q}) and (\ref{triv}), there is a projection onto $j(\R)/G$ over $M[G]$ from the product $R^0_{\lambda+1} \x R^1_{\lambda+1}$, where $R^0_{\lambda+1}$ is equivalent to $\Add(\kappa,j(\lambda))$ and $R^1_{\lambda+1}$ is the  ($\kappa^+$-closed)$^{M[G]}$ term forcing. In $V[(\tilde{G}^0 \x \tilde{G}^1) * g]$, there is a generic filter $H^0 \x H^1$ for $R^0_{\lambda+1} \x R^1_{\lambda+1}$ over $M[G]$, and $g$ is $\dot{\Q}^G$-generic over $M[G][H^0][H^1]$, with \beq M[G][H] \sub M[G][H^0][H^1].\eeq Since $\dot{\Q}^G$ lives in $H(\lambda)^{V[G]}$ and this is equal to $H(\lambda)^{M[G]}$, it is meaningful to write $M[G][g]$. We will analyse the relationship between $M[G][g]$ and $M[G][H^0][H^1][g]$ and argue that a hypothetical $\lambda$-Aronszajn tree $T$ in $M[G][g]$ cannot get a cofinal branch in $M[G][H^0][H^1][g]$, which contradicts the fact that $T$ does have a cofinal branch in $M[G][H][g]$ by (\ref{a:cof}). 

For the analysis, notice that $H^0,H^1,g$ are mutually generic\footnote{If $\dot{\Q}^G$ is $\kappa^+$-cc, the mutual genericity of $H^1$ and $g$ follows by Easton's lemma as argued in \cite{UNGER:1}. In any case, we can ensure that the generics are mutually generic by taking a generic for the product.} because $H^0 \x H^1 \x g$ is generic for the product $R^0_{\lambda+1} \x R^1_{\lambda+1} \x \dot{\Q}^G$ over $M[G]$ (recall that all these three forcings live in $M[G]$). It follows that we can rearrange them and have $M[G][H^0][H^1][g] = M[G][g][H^0][H^1]$, which gives us the inclusion relation \beq \label{a:key} M[G][g] \sub M[G][g][H^0][H^1].\eeq

(\ref{a}) follows by a standard argument which shows that $H^0 \x H^1$ cannot add a cofinal branch over $M[G][g]$ (use Fact \ref{f:spencer} applied over $M[G]$ to the $\kappa^+$-cc forcing $\dot{\Q}^G \x R^0_{\lambda+1}$ and $\kappa^+$-closed forcing $R^1_{\lambda+1}$).

(\ref{c}) Since $\dot{\Q}^G$ does not add new $\kappa$-sequences, $R^1_{\lambda+1}$ is still $\kappa^+$-closed over $M[G][g]$. The argument is finished by first noting that by Fact \ref{Kunen}, $R^0_{\lambda+1}$ cannot add a cofinal branch to $T$ over $M[G][g]$ because it is $\kappa^+$-cc here ($\dot{\Q}^G$ does not add new $\kappa$-sequences, which implies that $R^0_{\lambda+1}$ is even $\kappa^+$-Knaster in $M[G][g]$). Then we apply Fact \ref{f:spencer} over $M[G][g]$ to the $\kappa^+$-cc forcing $R^0_{\lambda+1}$ and $\kappa^+$-closed forcing $R^1_{\lambda+1}$ and conclude there is no cofinal branch in $T$ in $M[G][g][H^0][H^1]$ as desired.\footnote{Fact \ref{f:spencer} is actually not required for this case because $R^0_{\lambda+1}$ is $\kappa^+$-Knaster and hence $H^0$ does not add cofinal branches to $T$ over $M[G][g][H^1]$ by an argument of Baumgartner in \cite{BAUM:iter}.}
\end{proof}

\section{Open questions}\label{sec:open}
We conclude the paper by open questions.

{\bf Q1.} The methods used to prove Theorem \ref{th:improve} left open the question whether the tree property at $\kappa^{++}$ can be indestructible under all $\kappa^{+}$-cc forcing notions $\Q$ living in $V[\M]$. The key ingredient of the proof was the analysis using a projection from a product which requires that $\dot{\Q}$ can be ``grouped'' with Cohen part $\Add(\kappa,\lambda)$ of the Mitchell forcing $\M$. Notice that the question is also open with stronger forms of $\kappa^{+}$-cc.

{\bf Q2.} More specifically, it is open whether, or to what extent, Theorem \ref{th:improve} can be generalized to Prikry-type forcing notions with collapses. Depending on the setup of the Prikry forcing with collapses, the forcing itself does not live in $V[\Add(\kappa,\lambda)]$, but typically only in $V[\M]$ (this is for instance the case in \cite{FHS2:large} where the guiding generic exists only in $V[\M]$). However, the properties of the Prikry forcing might allow an analysis not in general available for an arbitrary $\kappa^+$-cc forcing in $V[\M]$.

{\bf Q3.} On a more general note, it seems worth studying the possible extent of the class of forcing notions which do not add $\lambda$-Aronszajn trees with respect to different models and different cardinals $\lambda$ at which the tree property holds (for instance at $\lambda = \kappa^+$ for a strong limit singular $\kappa$).

\medskip

\no {\bf Acknowledgements.} Both authors were supported by FWF/GA{\v C}R grant \emph{Compactness principles and combinatorics} (19-29633L).


\begin{thebibliography}{10}

\bibitem{ABR:tree}
Uri Abraham.
\newblock Aronszajn trees on $\aleph_2$ and $\aleph_3$.
\newblock {\em Annals of Pure and Applied Logic}, 24(3):213--230, 1983.

\bibitem{BAUM:iter}
J.~E. Baumgartner.
\newblock Iterated forcing.
\newblock In A.~R.~D. Mathias, editor, {\em Surveys in Set Theory}, pages
  1--59. Cambridge University Press, 1983.
\newblock London Math. Soc. Lecture Note Ser. 87.

\bibitem{F:u}
Andrew Brooke-Taylor, Vera Fischer, Sy-David Friedman, and Diana~C. Montoya.
\newblock Cardinal characteristics at $\kappa$ in a small
  $\mathfrak{u}(\kappa)$ model.
\newblock {\em Annals of Pure and Applied Logic}, 168(1):37--49, 2017.

\bibitem{CUMhandbook}
James Cummings.
\newblock Iterated forcing and elementary embeddings.
\newblock In Matthew Foreman and Akihiro Kanamori, editors, {\em Handbook of
  Set Theory}, volume~2. Springer, 2010.

\bibitem{CUMFOR:tp}
James Cummings and Matthew Foreman.
\newblock The tree property.
\newblock {\em Advances in Mathematics}, 133(1):1--32, 1998.

\bibitem{8fold}
James Cummings, Sy-David Friedman, Menachem Magidor, Assaf Rinot, and Dima
  Sinapova.
\newblock The eightfold way.
\newblock {\em The Journal of Symbolic Logic}, 83(1):349--371, 2018.

\bibitem{DS:graph}
Mirna {D{\v z}amonja} and Saharon Shelah.
\newblock Universal graphs at the successor of a singular cardinal.
\newblock {\em The Journal of Symbolic Logic}, 68(2):366--388, 2003.

\bibitem{FHS2:large}
Sy-David Friedman, Radek Honzik, and {\v S}{\'a}rka Stejskalov{\'a}.
\newblock The tree property at $\aleph_{\omega+2}$ with a finite gap.
\newblock Submitted, 2018.

\bibitem{FHS1:large}
Sy-David Friedman, Radek Honzik, and {\v S}{\'a}rka Stejskalov{\'a}.
\newblock The tree property at the double sucessor of a singular cardinal with
  a larger gap.
\newblock {\em Annals of Pure and Applied Logic}, 169:548--564, 2018.

\bibitem{GS:pol}
Shimon Garti and Saharon Shelah.
\newblock A strong polarized relation.
\newblock {\em The Journal of Symbolic Logic}, 77(3):766--776, 2012.

\bibitem{GS:small}
Shimon Garti and Saharon Shelah.
\newblock The ultrafilter number for singular cardinals.
\newblock {\em {Acta Math.\ Hungar.}}, 137(4):296--301, 2012.

\bibitem{GP:gap}
Mohammad Golshani and Alejandro Poveda.
\newblock The tree property at double successors of singular cardinals of
  uncountable cofinality with infinite gaps.
\newblock {Preprint on arXiv.}

\bibitem{HS:both-u}
Radek Honzik and {\v S}{\'a}rka Stejskalov{\'a}.
\newblock The tree property and the small ultrafilter number at regular or
  singular $\kappa$.
\newblock Work in progress.

\bibitem{JECHbook}
Tom{\'{a}\v{s}} Jech.
\newblock {\em Set Theory}.
\newblock Springer Monographs in Mathematics. Springer, Berlin, 2003.

\bibitem{JSkurepa}
Ronald Jensen and Karl Schlechta.
\newblock Result on the generic {K}urepa hypothesis.
\newblock {\em Archive for Mathematical Logic}, 30:13--27, 1990.

\bibitem{Singapore}
Ronald~B. Jensen.
\newblock Subcomplete forcing and {$\mathcal{L}$}-forcing.
\newblock In Chitat Chong, Qi~Feng, Theodore~A. Slaman, W.~Hugh Woodin, and Yue
  Yang, editors, {\em {$\mathcal{L}$}-Forcing, E-recursion, forcing and
  {$C^*$}-algebras}, pages 83--182. World Scientific, 2014.
\newblock Lecture Notes Series, Institute for Mathematical Sciences, National
  University of Singapore.

\bibitem{KANAMORIperfect}
Akihiro Kanamori.
\newblock Perfect-set forcing for uncountable cardinals.
\newblock {\em Annals of Mathematical Logic}, 19:97--114, 1980.

\bibitem{KUNENsat}
Kenneth Kunen.
\newblock Saturated ideals.
\newblock {\em The Journal of Symbolic Logic}, 43(1), 1978.

\bibitem{Kunen:new}
Kenneth Kunen.
\newblock {\em Set theory (Studies in Logic: Mathematical Logic and
  Foundations)}.
\newblock College Publications, 2011.

\bibitem{M:cof}
Menachem Magidor.
\newblock Changing cofinality of cardinals.
\newblock {\em Fundamenta Mathematicae}, 99:61--71, 1978.

\bibitem{NEEMAN:tp}
Itay Neeman.
\newblock The tree property up to $\aleph_{\omega+1}$.
\newblock {\em The Journal of Symbolic Logic}, 79(1):429--459, 2014.

\bibitem{SHELAHproper}
Saharon Shelah.
\newblock {\em Proper and Improper Forcing}.
\newblock Springer, 1998.

\bibitem{UNGER:1}
Spencer Unger.
\newblock Fragility and indestructibility of the tree property.
\newblock {\em Archive for Mathematical Logic}, 51(5--6):635--645, 2012.

\end{thebibliography}

\end{document}